\documentclass[11pt,twoside]{article}
\usepackage{latexsym}
\usepackage{amssymb,amsbsy,amsmath,amsfonts,amssymb,amscd}
\usepackage{amsthm}
\usepackage{epsfig, graphicx}

\newcommand{\commentout}[1]{}
\newcommand{\R}{\mathbb{R}}

\newcommand {\Chi} {{\bf \raise 2pt \hbox{$\chi$}} }

\newcommand {\F} { {\mathcal F} }

\newcommand {\f}   {\frac}

\newcommand{\beq}{\begin{equation}}
\newcommand{\beqa} {\begin{array}{rl}}
\newcommand{\eeq}{\end{equation}}
\newcommand{\eeqa}{\end{array}}
\newtheorem{theorem}{Theorem}[section]
\newtheorem{lemma}[theorem]{Lemma}

\newtheorem*{theoremEMRR}{Theorem \cite{EMRR05}}

\newcommand {\dx} {\,{\rm d}x}
\newcommand {\dy} {\,{\rm d}y}
\newcommand {\dz} {\,{\rm d}z}
\newcommand {\dmu} {\,{\rm d}\mu}
\title{\bf Exponential relaxation to self-similarity for the superquadratic fragmentation equation}

\author{P. Gabriel\thanks{University of Versailles St-Quentin-en-Yvelines, Laboratoire de Math\'ematiques de Versailles, CNRS UMR 8100, 45 Avenue de \'Etats-Unis, 78035~Versailles~cedex, France. Email: pierre.gabriel@uvsq.fr} \and F. Salvarani
\thanks{Dipartimento di Matematica F. Casorati, Universita degli Studi di Pavia, Via
Ferrata 1, 27100 Pavia, Italy. Email: francesco.salvarani@unipv.it}}

\date{}

\begin{document}
\maketitle
\pagestyle{plain}
\pagenumbering{arabic}

\begin{abstract}
We consider the self-similar fragmentation equation 
with a superquadratic fragmentation rate and provide a quantitative estimate of the spectral gap.
\end{abstract}

%%%%%%%%%%%%%%%%%%%%%%%%%%%%%%%%%%%%%%%%%%%%

\noindent{\bf AMS Class. No.} Primary: 35B40, 35P15, 45K05; Secondary: 82D60, 92D25

\

%%%%%%%%%%%%%%%%%%%%%%%%%%%%%%%%%%%%%%%%%%%%
\section{Introduction}
%%%%%%%%%%%%%%%%%%%%%%%%%%%%%%%%%%%%%%%%%%%%

The\emph{ fragmentation equation}
\beq\label{eq:frag}
  \left\{\begin{array}{ll}
    \partial_t\, f(t,x) = \F f(t,x),\qquad &t\geq0,\ x>0,
    \vspace{2mm}\\
    f(0,x) = f_\text{in}(x), \qquad &x > 0
  \end{array}\right.
\eeq
is a model that describes the time evolution of a population structured with respect to the size $x$ 
of the individuals.

The key term of the model is the fragmentation operator $\F$, defined as
\beq\label{eq:frag_operator} 
\F f (x):= \int_x^\infty b(y,x) f(y) \dy  - B(x) f(x).
\eeq
The fragmentation operator quantifies the generation of smaller individuals from a member of the population
of size $x$: the individuals split with a rate $B(x)$ and generate smaller individuals of size $y\in(0,x)$, whose
distribution is governed by the kernel $b(x,y)$.

Models involving the fragmentation operator appear in various applications.
Among them we can mention crushing of rocks, droplet breakup or combustion~\cite{Banasiak04} which are pure fragmentation phenomena, but also cell division~\cite{BP}, protein polymerization~\cite{GabrielPhD} or data transmission protocols on the web \cite{BCGMZ13}, for which the fragmentation process occurs together with some ``growth'' phenomenon.

In order to ensure the conservation of the total mass of particles which may occur during the fragmentation process, the coefficients $B(x)$ and $b(y,x)$ must be linked through the relation
\beq
\label{as:mass_cons}
\int_0^y xb(y,x)\dx=yB(y).
\eeq
This assumption ensures, at least formally, the mass conservation law
\beq
\label{eq:mass_cons}
\forall t>0,\qquad \int_0^\infty xf(t,x)\dx=\int_0^\infty xf_\text{in}(x)\dx:=\rho_\text{in}.
\eeq
Moreover, it is well known that $xf(t,x)$ converges to a Dirac mass at $x=0$ when $t\to+\infty$.
Usually, the various contributions that are available in the literature restrict their attention to coefficients which satisfy the homogeneous assumptions (see~\cite{EMRR05} for instance)
\beq
\label{as:coeff1}
B(x)=x^\gamma,\quad\gamma>0,\qquad \textrm{and}\qquad b(y,x)=y^{\gamma-1}p\Bigl(\f xy\Bigr),
\eeq
where $\dmu(z):= p(z)\dz$ is a positive measure supported on $[0,1]$
which satisfies 
$$
\int_0^1z\dmu(z)=1.
$$
These hypotheses guarantee that the relation~\eqref{as:mass_cons} is verified.

From a mathematical point of view, it is convenient to perform the (mass preserving) self-similar change of variable
\[
f(t,x)=(1+t)^{2/\gamma}g\Bigl(\f1\gamma\log(1+t),(1+t)^{1/\gamma}x\Bigr),
\]
or, by writing $g$ in terms of $f$,
\[
g(t,x):=e^{-2t}f(e^{\gamma t}-1,e^{-t}x).
\]
It allows to deduce that $g(t,x)$ satisfies the \emph{self-similar fragmentation equation}
\beq\label{eq:self-sim-frag}
  \left\{\begin{array}{l}
    \partial_t\, g + \partial_x (x g) + g= \gamma\F g,\qquad t\geq0,\ x>0,
    \vspace{2mm}\\
    g(0,x) = f_\text{in}(x), \qquad x > 0.
  \end{array}\right.
\eeq
Equation~\eqref{eq:self-sim-frag} belongs to the class of \emph{growth-fragmentation equations} and it
admits -- unlike Equation~\eqref{eq:frag} -- positive steady-states~\cite{DG10,EMRR05,PR05}.

Denote by $G$ the unique positive steady-state of Equation (\ref{eq:self-sim-frag}) with normalized mass,
i.e. the solution of
\[
(xG)'+G=\gamma\mathcal F G,\qquad G>0,\qquad \int_0^\infty xG(x)\dx=1.
\]
Then it has been proved (see \cite{EMRR05,MMP2}) that the solution $g(t,x)$ of the
self-similar fragmentation equation
(\ref{eq:self-sim-frag}) converges to $\rho_\text{in}G(x)$ when $t\to+\infty.$

Coming back to the fragmentation equation (\ref{eq:frag}), this result implies the convergence of $f(t,x)$ to the self-similar solution $(t,x)\mapsto\rho_\text{in}(1+t)^{2/\gamma}G((1+t)^{1/\gamma}x)$ and hence the convergence of $xf(t,x)$ to a Dirac mass $\delta_0$.

In order to obtain more precise quantitative properties of the previous equation, 
one can wonder about the rate of convergence of $g(t,x)$ to the asymptotic profile $G(x)$.
Many recent works are dedicated to this question and prove, under different assumptions and with different techniques, an exponential rate of convergence for growth-fragmentation equations~\cite{BCG,BCGMZ13,CCM10,CCM11,Cloez,LP09,PPS2,PR05}.

Nevertheless, to our knowledge the only results about the specific case of the self-similar fragmentation equation are those provided by C\'aceres, Ca\~nizo and Mischler~\cite{CCM10,CCM11}.
They prove exponential convergence in the Hilbert space $L^2((x+x^k)\dx)$ for a sufficiently large exponent $k$ in ~\cite{CCM11}, and in the Banach space $L^1((x^m+x^M)\dx)$ for suitable exponents $1/2<m<1<M<2$ in ~\cite{CCM10}.
For proving their results, the authors of the aforementioned articles require the measure $p$ to be a bounded function 
(from above and below) and the power $\gamma$ of the fragmentation rate to be less than $2$.

The current paper aims to obtain a convergence result for super-quadratic rates, namely when $\gamma \geq 2$.
We obtain exponential convergence to the asymptotic state by working in the weighted Hilbert space $L^2(x\dx)$, under
the following assumptions:
\beq
\label{as:coeff2}
\gamma\geq2\qquad\text{and}\qquad p(z)\equiv 2.
\eeq
The fact that $p(z)$ is a constant means that the distribution of the fragments is uniform: the probability to get a fragment of size $x$ or $x'$ from a particle of size $y>x,x'$ is exactly the same.
Then the condition $\int_0^1 zp(z)\dz=1$ imposes this constant to be equal to $2$, meaning that the fragmentation is necessarily binary.
Our assumption on $p$ is more restrictive than in~\cite{CCM10,CCM11}, but in return we get a stronger result in the sense that we obtain an estimate of the exponential rate.
Now we state the main theorem of this paper.
\begin{theorem}
\label{th:expo_decay}
Let $g_{\rm in}\in L^1(x\dx)\cap L^2(x\dx)$ and let $g\in C([0,\infty),L^1(x\dx))$ be the unique solution of the self-similar fragmentation equation
$(\ref{eq:self-sim-frag})$
with initial condition $g_{\rm in}$ and with fragmentation coefficients satisfying $\eqref{as:coeff1}$ and $\eqref{as:coeff2}$,
that is
$$
B(x)=x^\gamma,\quad\gamma\geq2\qquad \textrm{and}\qquad b(y,x)= 2y^{\gamma-1}.
$$
Then the following estimate holds:
\[
\|g(t,\cdot)-\rho_{\rm in}G\|_{L^2(x\dx)}\leq\|g_{\rm in}-\rho_{\rm in}G\|_{L^2(x\dx)}\,e^{-t},\qquad t\geq0.
\]
\end{theorem}

%%%%%%%%%%%%%%%%%%%%%%%%%%%%%%%%%%%%%%%%%%%%
\section{Preliminaries}
%%%%%%%%%%%%%%%%%%%%%%%%%%%%%%%%%%%%%%%%%%%%

Define the suitable weighted spaces
\[\dot L_k^p:=L^p(\R^+,x^k\,\dx)\quad\text{for}\ p\geq1,\ k\in\R,\qquad\text{and}\qquad\dot W_1^{1,1}:=W^{1,1}(\R^+,x\,\dx).\]
For $u\in\dot L^1_1$ we denote moreover by
\[M(x):=\int_0^x yu(y)\,\dy\]
 the primitive of $xu(x)$ which vanishes at $x=0.$

\

Now we recall the following existence and uniqueness result of a solution to the fragmentation equation, easily deduced from~\cite{EMRR05}, Theorems 3.1-3.2 and Lemma 3.4:

\begin{theoremEMRR}
If 
$$
B(x)=x^\gamma,\quad\gamma\geq2\qquad \textrm{and}\qquad b(y,x)= 2y^{\gamma-1},
$$
for any $f_{\rm in}\in \dot L^2_1\cap \dot L^1_1,$  there exists a unique solution $f\in C([0,\infty);\dot L^1_1)\cap L^1_{loc}([0,\infty);\dot L^1_{1+\gamma})$ to the fragmentation equation $(\ref{eq:frag})$ such that the mass conservation \eqref{eq:mass_cons} is satisfied.
If, moreover, $f_{\rm in}\in\Xi:= \dot L^1_{1+\gamma}\cap \dot W^{1,1}_1$, the associated solution $g$ to the self-similar fragmentation equation is such that
\[(g(t,\cdot))_{t\geq0}\ \text{is uniformly bounded in}\ \Xi.\]
\end{theoremEMRR}

\

In the following lemma we give some useful properties of the set 
$$
\Xi= \dot L^1_{1+\gamma}\cap \dot W^{1,1}_1
$$ 
and of the subset 
$$
\Xi_0:=\left\{u\in\Xi,\ \int_0^\infty xu(x)\,\dx=0\right\}.
$$

\begin{lemma}\label{lm:Xi}
The set $\Xi= \dot L^1_{1+\gamma}\cap \dot W^{1,1}_1$ satisfies 
\[
G\in\Xi,\qquad\Xi\subset C(0,\infty)\cap L^1\cap\dot L^2_1\cap\dot L^2_{1+\gamma}\qquad\text{and}
\]
\[
\forall u\in\Xi,\qquad\lim_{x\to0}xu(x)=\lim_{x\to+\infty}xu(x)=0.
\]
Moreover for any function $u\in\Xi_0$ the following inequality holds:
\[
\forall x>0,\qquad\left|M(x)\right|\leq x^{-\gamma}\|u\|_{\dot L_{1+\gamma}^1}.
\]
\end{lemma}

\begin{proof}
The fact that the steady-state $G$ belongs to $\Xi$ is a consequence of the estimates in~\cite{BCG}. In the case when $p(z)\equiv2$ it can also be deduced from the explicit formula (see~\cite{DG10})
\[
G(x)=\f{\gamma}{\Gamma(2/\gamma)}\,e^{-x^\gamma}
\]
where $\Gamma$ is the Euler Gamma function.

\

For $\gamma\geq0$ and $u\in\Xi$ such that $\int_0^\infty xu(x)\,\dx=0$ we can write for $x>0$
\[
|M(x)|=\left|-\int_x^\infty yu(y)\,\dy\right|\leq x^{-\gamma}\int_0^\infty y^{1+\gamma} |u(y)|\,\dy.
\]
\end{proof}

\

%%%%%%%%%%%%%%%%%%%%%%%%%%%%%%%%%%%%%%%%%%%%
\section{Proof of the main theorem}
%%%%%%%%%%%%%%%%%%%%%%%%%%%%%%%%%%%%%%%%%%%%

Define the self-similar fragmentation operator $\mathcal L u := -(x u)' -  u + \gamma\F u$ and denote by
\[
(u,v):=\int_0^\infty xu(x)v(x)\,\dx
\]
the canonical scalar product in $\dot L_1^2.$
Theorem~\ref{th:expo_decay} is a consequence of the following result.

\begin{theorem}\label{th:spectral_gap}
Under Assumptions \eqref{as:coeff1} and~\eqref{as:coeff2}, i.e.
for
$$
B(x)=x^\gamma,\quad\gamma\geq2\qquad \textrm{and}\qquad b(y,x)= 2y^{\gamma-1},
$$
we have
\[
\forall u\in\Xi_0,\qquad (u,\mathcal L u)\leq-\|u\|^2_{\dot L_1^2}.
\]
\end{theorem}

\begin{proof}
Using Lemma~\ref{lm:Xi} we can deduce, for $u\in\Xi_0$,
\[
(u,(x u)')=\int_0^\infty xu(x)\bigl(xu(x)\bigr)'\dx=
\frac{1}{2}\int_0^\infty \bigl((xu(x))^2\bigr)'\dx=0
\]
and
\begin{align*}
(u,\mathcal F u)&=2\int_0^\infty xu(x) \int_x^\infty y^{\gamma-1}u(y)\,\dy\,\dx-\int_0^\infty x^{\gamma+1}u^2(x)\,\dx\\
&=2\int_0^\infty x^{\gamma-1}u(x)\int_0^x yu(y)\,\dy\,\dx-\int_0^\infty x^{\gamma+1}u^2(x)\,\dx\\
&=2\int_0^\infty x^{\gamma-2}M'(x)M(x)\,\dx-\int_0^\infty x^{\gamma+1}u^2(x)\,\dx\\
&=-(\gamma-2)\int_0^\infty x^{\gamma-3}M^2(x)\,\dx-\int_0^\infty x^{\gamma+1}u^2(x)\,\dx\leq0.
\end{align*}
\end{proof}

\begin{proof}[Proof of Theorem~\ref{th:expo_decay}]
Assume first that $g_{\rm in}\in\Xi.$
From Theorem~\ref{th:spectral_gap} we obtain the differential inequality
\[\f{d}{dt}\|g(t,\cdot)-\rho_\text{in}G\|_{\dot L_1^2}\leq-\|g(t,\cdot)-\rho_\text{in}G\|_{\dot L_1^2}\]
which gives the result.
Then we may remove the additional assumption $g_\text{in}\in\Xi.$
\end{proof}

\

\section{Conclusion}

We have proved a spectral gap result for the self-similar fragmentation operator $\mathcal L$ with a superquadratic fragmentation rate $B(x).$
More precisely we have obtained that this spectral gap is larger than $1.$
This is a new result concerning the long-time behaviour of the fragmentation equation~\eqref{eq:frag}.
It also allows to extend the results obtained in~\cite{PG12} for non-linear growth-fragmentation equations.

\section*{Acknowledgement}
This work has been performed in the framework of the ANR project KIBORD, headed by Laurent Desvillettes.

%%%%%%%%%%%%%%%%%%%%%%%%%%%%%%%%%%%%
%
%%%%%% BIBLIO %%%%%%%%%%%%%%%%%%%%%%
%
%%%%%%%%%%%%%%%%%%%%%%%%%%%%%%%%%%%%

\end{document}